\documentclass[a4paper,11pt,english]{article}
\usepackage[english]{babel}
\usepackage{authblk}
\usepackage{amsmath}
\usepackage{amsthm}
\usepackage{anysize}
\usepackage{mathtools}
\usepackage{mathrsfs}
\usepackage{amssymb}
\usepackage{xcolor}
\usepackage{xfrac}
\usepackage{esint}
\usepackage{graphicx}
\usepackage{float}
\usepackage{array}
\usepackage{enumitem}
\usepackage{tikz-cd}
\usepackage{centernot}
\usepackage{stmaryrd}
\usepackage{comment}
\usepackage{bbm}

\usepackage{graphicx,color}
\usepackage{tikz,pgfplots,pgf}
\usepackage[font=small]{caption}

\numberwithin{equation}{section}

\newtheorem{theorem}{Theorem}[section]

\newtheorem{lemma}[theorem]{Lemma}
\newtheorem{proposition}[theorem]{Proposition}

\DeclareMathOperator{\loc}{loc}

\newcommand{\field}[1] {\mathbb{#1}}
\newcommand{\N}{\field{N}}

\newcommand{\R}{\field{R}}

\def\a{\alpha}

\def\e{\varepsilon}
\def\D{\Delta}

\def\g{\gamma}

\def\l{\lambda}

\def\O{\Omega}
\def\p{\partial}

\def\s{\sigma}

\def\v{\varphi}

\def\ua{\uparrow}

\newcommand{\mc}{\mathcal}
\newcommand{\mf}{\mathfrak}

\DeclareMathOperator{\dist}{dist}

\makeatletter
\@namedef{subjclassname@2020}{\textup{2020} Mathematics Subject Classification}
\makeatother

\author[1]{Juli\'{a}n L\'{o}pez-G\'{o}mez \thanks{jlopezgo@ucm.es}}
\author[2]{Juan Carlos Sampedro \thanks{juancarlos.sampedro@upm.es}}
\affil[1]{\small Universidad Complutense de Madrid\\ Instituto de Matem\'{a}tica Interdisciplinar\\ Department of Mathematical Analysis and Applied Mathematics \\ Plaza de las Ciencias 3\\ 28040 Madrid, Spain}
\affil[2]{\small Universidad Polit\'ecnica de Madrid\\ Instituto de Matem\'{a}tica Interdisciplinar \\ Departamento de Matem\'atica Aplicada a la Ingenier\'ia Industrial\\ Ronda de Valencia 3\\ 28012 Madrid, Spain}

\title{\textbf{Blow-up estimates and a priori bounds for the positive solutions of a class of superlinear indefinite elliptic problems}\footnote{This work has been supported by the Ministry of Science and Innovation of Spain under the 
Research Grant PID2021-123343NB-I00 and by the Institute of Interdisciplinary Mathematics of the Complutense University of Madrid}}

\begin{document}

\maketitle

\begin{abstract}
\noindent In this paper  we find out  some new blow-up estimates for the positive explosive solutions of a
paradigmatic class of  elliptic boundary value problems of superlinear indefinite type. These estimates are obtained by combining the scaling technique of Guidas--Spruck together with a  generalized De Giorgi--Moser weak Harnack inequality found, very recently, by  Sirakov \cite{SV0,SV}. In a further step, based on a comparison result of Amann and L\'opez-G\'omez
\cite{ALG},  we will show how these bounds provide us with some sharp a priori estimates for the classical positive solutions of a wide variety of superlinear indefinite problems. It turns out that this is the first general result where the decay rates of the potential in front of the nonlinearity ($a(x)$ in \eqref{1.1}) do not play any role for getting a priori bounds for the positive solutions when $N\geq 3$.
\end{abstract}



\section{Introduction}

In this paper we deliver some new blow-up estimates and universal a priori bounds for the positive solutions of the semilinear elliptic boundary value problem
\begin{equation}
	\label{1.1}
	\left\{\begin{array}{ll}
	\mathscr{L} u=\l u+a(x)u^r & \quad \hbox{in}\;\; \Omega, \\[1ex]
	\mathscr{B}u=0 & \quad \hbox{on}\;\; \p \Omega,
	\end{array}\right.
\end{equation}
where $\Omega$ is a bounded domain of $\mathbb{R}^{N}$, $N\geq 1$, of class $\mc{C}^{2}$, $\l\geq 0$, and
$$
  \mathscr{L}u=-\text{div}(A(x) \nabla u),
$$
$A$ being a matrix of order $N$ with entries in $W^{1,\infty}(\O)$ such that, for some $\mu>0$ and
every $\xi\in \R^{N}$ and $x\in\O$,
$$
    \langle A(x) \, \xi, \xi \rangle \geq \mu \, |\xi|^2 \equiv \mu \sum_{j=1}^N \xi_j^2.
$$
Thus, $\mathscr{L}$ is uniformly elliptic in $\O$. Moreover, in \eqref{1.1}, the boundary of $\O$, $\partial\O$, is divided  into two disjoint open and closed subsets $\Gamma_{0}$, $\Gamma_{1}\subset \partial\O$ where the boundary operator
$\mathscr{B}: \mc{C}(\Gamma_{0})\otimes \mc{C}^{1}(\Gamma_{1})\to \mc{C}(\partial\O)$ is defined through
\begin{equation}
\label{1.2}
\mathscr{B} u:=\left\{
	\begin{array}{ll}
		u & \text{ on } \Gamma_{0}, \\[1ex]
		\partial_{\nu}u+\beta(x)u & \text{ on } \Gamma_{1},
	\end{array}
	\right.
\end{equation}
where $\nu\in\mc{C}(\Gamma_{1},\R^{N})$ is an outward pointing nowhere tangent vector field,  and $\beta\in \mc{C}(\Gamma_{1})$, $\beta\geq 0$. Hence, $\mathscr{B}$ is the Dirichlet boundary operator on $\Gamma_{0}$ and the Neumann or a first order regular oblique derivative boundary operator on $\Gamma_{1}$.
In \eqref{1.1}, we also suppose that $a\in \mc{C}(\bar{\O})$ and $r>1$.
Throughout this paper, setting
\begin{equation*}
	\O_{+}:=\{x\in\O \, : \, a(x)>0\}, \quad \O_{-}:=\{x\in \O \, : \, a(x)<0\}, \quad \O_{0}:=
   a^{-1}(0),
\end{equation*}
we will assume that $\O_{+}$ and $\O_{-}$ are of class $\mc{C}^{2}$, and that
$$
   \O_{+}\neq \emptyset.
$$
As $a\in\mc{C}(\bar\O)$,  $\O_{+}$ and $\O_{-}$ are open subsets of $\O$, while
$\O_0$ is a compact subset of $\bar \O$. Moreover, $a$ is bounded away from zero on compact subsets of $\O_{\pm}$. Since $\O_{+}$ and $\O_{-}$ are of class $\mc{C}^{2}$, their boundaries  can have, at most, finitely many components. Finally, we require that, for every component $\Gamma$ of $\Gamma_{1}$,
\begin{equation}
\label{1.3}
\Gamma\subset \partial\O_\pm \quad \hbox{if}\;\; \Gamma\cap \partial\O_\pm \neq \emptyset.
\end{equation}
Fixed $p>N$, by a solution of \eqref{1.1}, we mean a pair $(\lambda,u)\in \mathbb{R}_{+} \times  W^{2,p}(\Omega)$ such that $u$ satisfies \eqref{1.1} almost everywhere. A positive solution of \eqref{1.1} is a solution $(\lambda,u)\in \R_{+} \times W^{2,p}(\Omega)$ such that $u\geq 0$ in $\O$ but $u\neq 0$ (i.e., $u\gneq 0$). Since $r>1$, by the maximum principle, the positive solutions satisfy $u\gg 0$, in the sense that $u(x)>0$ for all $x\in\Omega$ and $\frac{\p u}{\p \nu}(x)<0$ for all $x\in u^{-1}(0)\cap \p\O$. This concept of solution makes sense because,  by the Sobolev embeddings,  $W^{2,p}(\O) \hookrightarrow \mc{C}^{2-N/p}(\bar{\O})$ if $p>N$. So, the boundary operator
$\mathscr{B}$  acts on the classical sense.
\par
As $\O_{+}\neq \emptyset$, this setting is wide enough as to include purely superlinear problems ($\O=\O_{+}$), general superlinear problems ($\O_{-}=\emptyset$) as well as  superlinear indefinite problems ($\O_{-}\neq \emptyset$). In this paper we will focus attention on superlinear indefinite problems and on general superlinear problems without specific boundary conditions.
\par
The problem of the existence of a priori bounds for purely superlinear  problems ($\O=\O_{+}$) has attracted a huge amount of attention since the paper of Br\'ezis and Turner \cite{BT} established
the existence of a priori bounds, through some Hardy--Sobolev type inequalities, for the positive weak solutions of
\begin{equation}
	\label{1.4}
	\left\{\begin{array}{ll}
	-\D u=\l u+a(x)u^r & \quad \hbox{in}\;\; \Omega, \\[1ex]
	u=0 & \quad \hbox{on}\;\; \partial\O,
	\end{array}\right.
\end{equation}
provided
$$
  r<p_{BT}\equiv \frac{N+1}{N-1}.
$$
Later,  Gidas and Spruck \cite{GS} got the existence of a priori bounds for the classical positive solutions of \eqref{1.4} for all $r>1$ if $N=2$, and for
$$
    r<p_{GS}\equiv \frac{N+2}{N-2} \quad  \text{if } N\geq 3.
$$
On this occasion the proof consisted of a blowing-up technique reducing the original equation to a Liouville type problem. Almost simultaneously, the same result was found by De Figueiredo, Lions and Nussbaum \cite{FLN} by exploiting the symmetry properties of the Laplace operator.
\par
Although the critical exponent $p_{BT}$ might be thought to
be of a technical nature, Souplet \cite{S} has recently shown that it is actually optimal for weak solutions. Indeed, according to \cite{S}, whenever $p > p_{BT}$, there is a suitable nonlinearity $f(x,u) = a(x)u^{p}$, with $a(x) \geq 0$ and $a \in L^{\infty}(\O)$, for which \eqref{1.4} admits a positive unbounded weak solution. \par
More recently, Sirakov \cite{SV0} has generalized these classical results obtaining the following theorem with no need to impose boundary conditions on the whole domain. Our statement is an adaptation of \cite[Th. 2]{SV0} to our setting. It should be noted that this result applies to  general superlinear problems in the sense that, in contrast to the classical results, $a(x)$ is allowed to vanish somewhere in $\O$. Subsequently, we will denote by $\mc{P}_\l$ the $\l$-projection operator, $\mc{P}_\l(\l,u)=\l$.

\begin{theorem}[\textbf{Sirakov, 2020}]
	\label{SI}
Suppose $\O_{-}=\emptyset$ and $\Gamma\subset \partial\O$ is a  relatively open subset of $\partial\O$
of class $\mc{C}^{1,1}$ regularity.  Let $\mc{D}$ be a subdomain of $\R^N$ such that $\bar{\mc{D}}\subset \Omega\cup\Gamma$ and either
\begin{equation*}
	r<\frac{N+1}{N-1}, \quad \hbox{or} \;\;  r<\frac{N}{N-2} \;\; \hbox{if} \;\; \bar{\mc{D}}
 \subset \O.
\end{equation*}
Then, for every subset of positive solutions, $\mathscr{S}\subset \R_{+}\times W^{2,p}(\O)$, of
\begin{equation}
		\label{1.5}
		\left\{\begin{array}{ll}
			\mathscr{L} u=\l u + a(x)u^r & \quad \hbox{in}\;\; \Omega, \\[1ex]
			u=0 & \quad \hbox{on}\;\; \Gamma,
		\end{array}\right.
\end{equation}
with $\l$-projection $\Lambda_{\mathscr{S}} \equiv \mc{P}_\l(\mathscr{S})$ bounded,  there exists
a constant $C>0$ such that
$$
    \sup_{(\l,u)\in \mathscr{S}}\|u\|_{L^{\infty}(\mc{D})}\leq  C.
$$
\end{theorem}

These bounds are derived from the  generalized  weak Harnack inequality of De Giorgi and Moser type
given  by Sirakov in \cite[Cor. 1.1]{SV}. Next, by their relevance in this paper, we are going to
deliver a short summary of these results at the light of the perspective adopted in  \cite{LG13}.
The classical  De Giorgi--Moser weak Harnack inequality
establishes that, for every $q\in [1,\tfrac{N}{N-2})$, $y\in\O$ and $R>0$ with $\bar B_{R}(y)\subset \O$, any weak non-negative superharmonic function $u\in H^{1}(\O)\equiv W^{1,2}(\O)$ satisfies
\begin{equation}
	\label{1.6}
 \|u\|_{L^{q}(B_{R}(y))}\leq C \inf_{B_{R/2}(y)} u,
\end{equation}
where $C=C(N,q,R)$. On the other hand, by the uniform decay property of Hopf and Walter (see, e.g., \cite[Th. 1.5]{LG13}), there exists $M=M(R)>0$ such that, for every positive superharmonic $u\in W^{2,p}(\O)$,
\begin{equation}
	\label{1.7}
u(x)\geq M \Big(\inf_{B_{R/2}(y)}u \Big)\dist(x, \partial B_{R}(y)).
\end{equation}
Combining \eqref{1.6} and \eqref{1.7}, it follows the estimate
\begin{equation}
	\label{1.8}
	u(x)\geq C  \|u\|_{L^{q}(B_{R}(y))} \dist(x, \partial B_{R}(y)),
\end{equation}
where $C=C(N,q,R)$. Actually, \eqref{1.8} holds by replacing $B_{R}(y)$ by any subdomain $\mc{D}$ with
$\bar{\mc{D}} \subset \O$. The Sirakov result extends \eqref{1.8} up to the boundary $\p\O$ by establishing that, for every weak non-negative superharmonic function $u\in H^{1}(\O)$ and any $q\in [1,\frac{N}{N-1})$,
\begin{equation}
\label{1.9}
	\|u\|_{L^{q}(\O)}\leq C  \inf_{\O}\frac{u}{d}, \qquad d(x)=\dist(x,\p\O),
\end{equation}
where $C=C(N,q)$. The exponent $\tfrac{N}{N-1}$ is optimal in the sense that this estimate fails
if $q\geq \tfrac{N}{N-1}$. From \eqref{1.9},  Sirakov \cite{SV} infers the following result.

\begin{theorem}[\textbf{Sirakov, 2022}]
	\label{SK}
	Let $u\in H^{1}(\O)$ be a weak non-negative superharmonic function. Then, there exists a positive constant $C>0$, independent of $u$, such that, for every   $q\in[1,\tfrac{N}{N-1})$,
\begin{equation*}
		u(x)\geq C  \|u\|_{L^{q}(\O)} \dist(x,\partial\O)\quad \hbox{for all}\;\; x\in\O.
\end{equation*}
\end{theorem}

By combining Theorem \ref{SK} with some sharp refinements of the blowing-up arguments of
Gidas and Spruck \cite{GS}, we will infer the first main result of this paper, which can be stated as follows.
It establishes the blow-up rates of the explosive positive solutions of \eqref{1.1}.

\begin{theorem}
	\label{BUS}
Suppose that $\O_+\neq \emptyset$, $\bar\O_+\subset \O$, $\O_{-}\neq \emptyset$, $1\leq q<\tfrac{N}{N-1}$, and $\{(\l_{n},u_{n})\}_{n\in \N}\subset \R_{+}\times W^{2,p}(\O)$ is  a sequence of positive solutions of \eqref{1.1}, with $\{\l_{n}\}_{n\in\N}$
bounded, such that
$$
  \lim_{n\to+\infty}\|u_{n}\|_{L^{\infty}(\O_{+})}= +\infty.
$$
Then, there exist $C= C(N,q)>0$ and $n_0\in \N$, such that, for every $n\geq n_0$,
\begin{equation}
		\label{1.10}
		u_{n}(x) \geq C \|u_{n}\|_{L^{\infty}(\O_{+})}^{1+\frac{(1-r)N}{2q}} \dist(x,\partial\O_{+}) \quad
\hbox{for all}\;\; x\in\O_{+}.
\end{equation}
\end{theorem}

The estimate \eqref{1.10} is of great interest on its own. It might be useful even to get some parabolic counterparts in the vein outlined by Quittner and Simondon in \cite{QS}.
\par
In the context of superlinear indefinite problems, where $\O_+\neq \emptyset$ and $\O_{-}\neq \emptyset$, no effective method for obtaining a priori bounds had been known, at least,  until Berestycki, Capuzzo-Docetta and Nirenberg \cite{BCN,BCN2} provided a sufficient condition in terms of the behaviour of $a(x)$ near $\O_{0}$. Precisely, imposing that
\begin{equation}
	\label{1.11}
a\in\mc{C}^{2}(\bar\O), \quad \nabla a(x)\neq 0 \;\; \hbox{for all}\;\;  x\in\bar\O_+\cap \bar \O_-\subset \O,
\end{equation}
the main result of \cite{BCN} guarantees the existence of a priori bounds for the classical positive solutions of \eqref{1.1} in the special case when  $\mathscr{L}=-\Delta$, i.e., for $A(x)\equiv I_{\R^N}$,  under Dirichlet boundary conditions,  provided that
$$
   1<r<\frac{N+2}{N-1}.
$$
Later, Chen and Li \cite{CL,CL2} proved, by a novel moving plane argument, that actually \eqref{1.11} entails
the existence of a priori bounds for the positive solutions of \eqref{1.1} in a neighborhood of $\O_{0}$ for every $r > 1$.
\par
Under the general assumptions of this paper, adapting the Liouville theorem of \cite{BCN} and using some
comparison arguments, Amann and L\'{o}pez-G\'{o}mez \cite{ALG} improved the results of \cite{BCN} in a different direction. Their main result reads as follows in our present setting.

\begin{theorem}[\textbf{Amann and L\'{o}pez-G\'{o}mez, 1998}]
	Suppose  $\O_+\neq \emptyset$, $\bar\O_+\subset\O$, $\O_{-}\neq \emptyset$, and  there exist a continuous function  $\alpha:\bar{\O}_{+}\to\R_{+}$, with $\a(x)>0$ for all $x\in \p\O_+$, and a constant $\gamma > 0$ such that
\begin{equation}
\label{1.12}
a(x)=\alpha(x)[\dist(x,\partial\O_{+})]^{\gamma} \quad \hbox{for all}\;\; x\in\O_{+}.
\end{equation}
Suppose, in addition, that
\begin{equation}
\label{1.13}
   r<\frac{N+1+\gamma}{N-1} \quad  \text{ and } \quad r<\frac{N+2}{N-2} \quad \text{ if } N\geq 3.
\end{equation}
	Then, if $\mathscr{S}\subset \R_{+}\times W^{2,p}(\O)$ is a set of positive solutions of \eqref{1.1} such that $\Lambda_{\mathscr{S}}$ is bounded, there exists a constant $C>0$ such that
$$
\sup_{(\l,u)\in \mathscr{S}}\|u\|_{L^{\infty}(\O)}\leq C.
$$
\end{theorem}

Note that \eqref{1.13} reduces to the Gidas--Spruck condition $r<\frac{N+2}{N-2}$ if $N\geq 3$
 provided that $\g\geq \frac{2N}{N-2}$. Some time later, Du and Li \cite{DL} were also able to
 get the a priori bounds up to $r<\frac{N+2}{N-2}$ if $N\geq 3$ by imposing the additional conditions that
$\mathscr{L}=-\D$ and that
 \begin{equation}
	\label{1.14}
	a(x)=\alpha_{1}(x)[\dist(x,\partial\O_{+})]^{\gamma_{1}} \quad \hbox{for all}\;\; x\in\O_{-},
\end{equation}
for a continuous function $\a_1:\bar\O_-\to \R_+$ such that $\a_1(x)>0$ if $x\in\p\O_-$, and
a constant $\g_1>0$.

A question that arises rather naturally is to ascertain whether or not these decaying conditions on $a(x)$ ---\eqref{1.12} and \eqref{1.14}--- are really necessary to guarantee the existence of a priori bounds. The main result of this paper shows that the decaying conditions \eqref{1.12} and \eqref{1.14} are unnecessary
to have a priori bounds as soon as $r<p_{BT}=\frac{N+1}{N-1}$ if $N\geq 3$. It reads as follows.

\begin{theorem}
	\label{APB}
	Suppose that $\O_{+}\neq \emptyset$, $\bar\O_+\subset\O$, $\O_{-}\neq \emptyset$, and
$$
r<p_{BT}:=\frac{N+1}{N-1} \quad \hbox{if} \;\; N\geq 3.
$$
	Then, for every set of positive solutions of \eqref{1.1}, $\mathscr{S}\subset \R_{+}\times W^{2,p}(\O)$,
with $\Lambda_{\mathscr{S}}$ bounded, there exists a positive constant $C>0$ such that
$$
  \sup_{(\l,u)\in \mathscr{S}}\|u\|_{L^{\infty}(\O)}\leq C.
$$
\end{theorem}

Therefore, by simply imposing $r<p_{BT}$ if $N\geq 3$, one can get universal a priori bounds
regardless the decay rates of $a(x)$ on $\p\O_+\cup \p\O_-$ and the nature of the boundary conditions.
We stress again that, according to Souplet \cite{S}, whenever $p > p_{BT}$, there exists $a \gneq 0$ and $a \in L^{\infty}(\O)$ such that \eqref{1.4} admits a positive unbounded weak solution. Thus, the weak counterpart of Theorem \ref{APB}, is optimal if valid. Anyway, it remains an open problem to ascertain whether or not one might improve the range
$r<p_{BT}$ up to reach $p_{GS}:=\frac{N+2}{N-2}$ if $N\geq 3$ under the minimal requirements of Theorem \ref{APB}. Our proof of Theorem \ref{APB} proceeds by contradiction from  the blowing-up rates of Theorem \ref{BUS}.
\par
This paper is organized as follows. Section \ref{Se2} collects some preliminary results that will be used in the proofs of the main results of this paper. Section \ref{Se3} proves Theorem \ref{BUS}. Finally, based on these bounds, Section \ref{Se4} delivers the proof of Theorem \ref{APB}.

\section{Preliminaries}\label{Se2}

We begin by delivering a technical lemma necessary to apply the scaling arguments of Gidas--Spruck type.

\begin{lemma}
	\label{L1}
    Let $\O$ be an open  subset of $\R^{N}$ and $\{z_{n}\}_{n\in\mathbb{N}}\subset \mathbb{R}^{N}$, $\{\mu_{n}\}_{n\in\mathbb{N}}\subset (0,+\infty)$ be two sequences satisfying
	\begin{equation}
		\label{L101}
		\lim_{n\to+\infty} z_{n}=z_{0}\in -\O,  \quad
		\lim_{n\to+\infty}\mu_{n}=+\infty.
	\end{equation}
	Then, for every $R>0$, there exists $n_0=n_0(R)\in\mathbb{N}$ such that
	\begin{equation}
		\label{L102}
		B_R(0)\subset \mu_{n}\left( z_{n}+\O\right) \quad \hbox{for all}\;\; n\geq n_0.
	\end{equation}
Therefore,
	\begin{equation*}
		\bigcup_{n\in\mathbb{N}}\mu_{n}\left( z_{n}+\O\right)=\mathbb{R}^{N}.
	\end{equation*}
\end{lemma}
\begin{proof}
Since $z_{0}\in -\O$ and $-\O$ is open, there exists $\varepsilon>0$ such that $B_\varepsilon(z_{0})\subset -\O$. Pick $R>0$ and $y\in\mathbb{R}^{N}$ satisfying $|y|< R$, and observe that
\begin{equation}
		\label{L11}
		y\in \mu_{n}\left( z_{n}+\O\right) \quad \Longleftrightarrow \quad z_{n}-\mu_{n}^{-1}y\in -\O.
\end{equation}
	Choose $n_0\in\mathbb{N}$ such that
	\begin{equation}
		R \, \mu_{n}^{-1}<\frac{\varepsilon}{2}, \quad |z_{n}-z_{0}|<\frac{\varepsilon}{2}, \quad \hbox{for
all}\;\; n\geq n_0.
	\end{equation}
Then, for every $n\geq n_0$,
	\begin{align*}
		|z_{n}-\mu_{n}^{-1}y-z_{0}|\leq \mu_{n}^{-1} |y| + |z_{n}-z_{0}| <  R \mu_{n}^{-1}+ |z_{n}-z_{0}|<\varepsilon,
	\end{align*}
which implies that
	\begin{equation}
		z_{n}-\mu_{n}^{-1}y\in B_\e(z_{0})\subset -\O
	\end{equation}
for all $n\geq n_0$. Thus,  by \eqref{L11}, $y\in \mu_{n}(z_{n}+\O) $ for all $n\geq n_0$. As $n_0$ depends on $R$ but not on $y$, we can infer \eqref{L102}. This concludes the proof.
\end{proof}

In the proof of Theorem \ref{APB} we will invoke the next result, which is \cite[Th. 4.1]{ALG} and it is  based on the strong maximum principle.

\begin{proposition}
	\label{ThPM}
	If $\mathscr{S}\subset \R_{+}\times W^{2,p}(\Omega)$ is a set of positive solutions of \eqref{1.1}, then
	\begin{equation}
		\sup_{(\l,u)\in\mathscr{S}}\sup_{x\in \O_{+}} u(x) < +\infty \quad \Longrightarrow \quad \sup_{(\l,u)\in\mathscr{S}}\sup_{x\in \O} u(x) < +\infty.
	\end{equation}
	In other words, uniform a priori bounds on $\O_{+}$ imply uniform bounds on $\O$.
\end{proposition}

The next result, which is \cite[Lem. 4.2]{ALG}, will be used in the proof of Theorem \ref{BUS}.
It establishes that the blow up of the positive solutions of \eqref{1.1} must occur on $\p \O_{+}$.

\begin{theorem}
	\label{ThFr}
	Suppose that $r<\tfrac{N+2}{N-2}$ if $N\geq 3$.
	Let $\{u_{n}\}_{n\in \N}\subset W^{2,p}(\O)$ be a sequence of positive solutions of \eqref{1.1}
	such that
	$$\lim_{n\to+\infty}\|u_{n}\|_{L^{\infty}(\O_{+})}= +\infty,$$
and let  $\{x_{n}\}_{n\in\N}\subset \bar{\O}_{+}$ be such that $u_{n}(x_{n})=\|u_{n}\|_{L^{\infty}(\O_{+})}$ for each $n\in \N$. Then,
	$$\lim_{n\to+\infty}\dist(x_{n},\partial\O_{+})=0.$$
\end{theorem}

The final preliminary result, which is \cite[Th. 5]{HKT}, will be used in the proof of Theorem \ref{BUS}. It establishes the measure density condition for regular domains.

\begin{theorem}
	\label{Th2.3}
	Let $\O$ be a $\mc{C}^{1}$-subdomain of $\R^{N}$. Then, there exists a constant $c>0$ such that, for every $x\in\O$ and $r\in (0,1]$,
$$
\mathbf{m}(B_{r}(x)\cap \O)\geq c  r^{N},
$$
where $\mathbf{m}$ stands for the Lebesgue measure of $\R^{N}$.
\end{theorem}

\section{Proof of Theorem \ref{BUS}}\label{Se3}

Suppose $1\leq q<\tfrac{N}{N-1}$, and choose $\{x_{n}\}_{n\in\N}\subset \bar{\O}_{+}$ with $$
   u_{n}(x_{n})=\|u_{n}\|_{L^{\infty}(\O_{+})}\quad \hbox{for all}\;\; n\in \N.
$$
By hypothesis,
	\begin{equation}
\label{SAMP}
		\lim_{n\to+\infty}\|u_{n}\|_{L^{\infty}(\Omega_{+})}=+\infty.
	\end{equation}
Since $\{x_{n}\}_{n\in\mathbb{N}}\subset \bar\Omega_+$ and $\{\l_{n}\}_{n\in \N}$ is bounded, by compactness, there exist two subsequences, again denoted by $\{x_{n}\}_{n\in\mathbb{N}}$  and $\{\l_{n}\}_{n\in\N}$, such that
	\begin{equation*}
		\lim_{n\to+\infty}x_{n}=x_{0}, \quad \lim_{n\to +\infty}\l_{n}=\l_{0},
	\end{equation*}
for some $x_{0}\in\bar{\Omega}_+$ and $\l_{0}\geq 0$. By Theorem \ref{ThFr}, necessarily, $x_{0}\in\partial\O_{+}$ and consequently $a(x_{0})=0$.  Subsequently, we set  $M_{n}:=\|u_{n}\|_{L^{\infty}(\O_{+})}$ for each $n\in \N$, and
consider the scaling function
	\begin{equation}
		\label{K}
		v_{n}\left(\nu_{n}^{-1}(x-x_{n})\right)=\nu_{n}^{\frac{2}{r-1}} \, u_{n}(x), \quad x\in\Omega, \quad \nu_{n}:=M_{n}^{\frac{1-r}{2}}.
	\end{equation}
	Observe that $\nu_{n}\to 0$ as $n\to+\infty$ and  that, for every  $n\in\N$,
	\begin{equation*}
		v_{n}(0)=M_{n}^{-1}u_{n}(x_{n})=1.
	\end{equation*}
According to \eqref{K}, for every $n\geq 1$, the function $v_n$ is well defined in
the region
$$
  \mc{D}_n := \nu_n^{-1}(-x_n+\O),
$$
and $v_{n}$ can be equivalently defined by
 \begin{equation}
 	v_{n}(y):=\nu_{n}^{\frac{2}{r-1}}u_{n}(x_{n}+\nu_{n}y), \qquad y\in \mc{D}_n.
 \end{equation}
 After some straightforward manipulations, it is easily seen that $v_{n}\in W^{2,p}(\mc{D}_{n})$ and that it satisfies
\begin{equation}
		\mathscr{L} v_{n}=\l_{n} \nu_{n}^{2} v_{n} + a(x_{n}+\nu_{n}y) \, v_{n}^r, \quad y\in \mc{D}_{n}.
\end{equation}
In order to understand the behavior of the sets $\mc{D}_{n}$ as $n \to +\infty$,  we apply Lemma \ref{L1} with
$$
   z_{n}:=-x_{n}, \quad \mu_{n}:=\nu_{n}^{-1}, \quad n\in \N,
$$
to deduce that, for every $R>0$,  there exists $n_0\in\N$ such that
\begin{equation}
	\label{CD}
	B_{R}(0)\subset \nu_{n}^{-1}(-x_{n}+\O)=\mc{D}_{n}, \quad \text{ for all } n\geq n_0.
\end{equation}
Fix $R>0$ and choose $n_0\in \N$ satisfying \eqref{CD}. Then, $v_{n}$ is defined on $B_R\equiv B_{R}(0)$ for each $n\geq n_0$ and it satisfies
\begin{equation}
\label{3.6}
	\left\{
	\begin{array}{ll}
		\mathscr{L} v_{n}=\l_{n}\nu_{n}^{2} v_{n} + a(x_{n}+\nu_{n}y) \, v_{n}^r, & y\in B_{R}(0), \\
			0\leq v_{n}\leq v_{n}(0)=1. &
	\end{array}
	\right.
\end{equation}
By the elliptic regularity $L^p$-theory, we can infer from \eqref{3.6} that $\{v_{n}\}_{n\geq n_0}$ is bounded in $W^{2,p}(B_{R})$ for all $p>N$. Passing to a suitable subsequence, again denoted by $\{v_{n}\}_{n\in\N}$, and using the compactness of the embedding of $ W^{2,p}(B_{R})$ in each of the spaces $\mc{C}^{1}(B_{R})$, $W^{1,p}(B_{R})$ and $L^{p}(B_{R})$, we can assume that there exists $v\in W^{2,p}(B_{R})$ such that $v\geq 0$, $v_{n}\rightharpoonup v$ in $W^{2,p}(B_{R})$ and $v_{n}\to v$ in $W^{1,p}(B_{R})$ as well as in $\mc{C}^1(B_{R})$. This implies that
$$
   \mathscr{L} v_{n}\rightharpoonup -\text{div}(A(x_{0}) \nabla v) \quad \hbox{in}\;\;  L^{p}(B_{R}),
$$
and that
\begin{equation}
	\lim_{n\to+\infty}\left[ \l_{n}\nu_{n}^{2}v_{n} + a(x_{n}+\nu_{n}y) v_{n}^r\right] =a(x_{0})\, v^{r}
\end{equation}
strongly in $L^{p}(B_{R})$. Since $a(x_{0})=0$ and $a:\bar{\O}\to\R$ is continuous, necessarily
$$
  \lim_{n\to +\infty} a(x_{n}+\nu_{n}y)= a(x_{0})=0.
$$
Hence, $v\in W^{2,p}(B_{R})$ satisfies
\begin{equation}
		\left\{
	\begin{array}{ll}
		-\text{div}(A(x_{0}) \nabla v)=0, \quad y\in B_{R}, \\[1ex]
		0\leq v\leq v(0)=1.
	\end{array}
	\right.
\end{equation}
As this argument can be repeated for each $R>0$, by a standard diagonal sequence argument in $R>0$ it is not difficult to see that $v$ can be extended to a function $v\in W^{2,p}_{\loc}(\R^{N})$ such that
\begin{equation}
	\left\{
	\begin{array}{ll}
		-\text{div}(A(x_{0}) \nabla v)=0, \quad y\in \R^{N}, \\[1ex]
		0\leq v\leq v(0)=1.
	\end{array}
	\right.
\end{equation}
Also by elliptic regularity, we may infer that $v\in \mc{C}^{2}(\R^{N})$. Moreover, there exists an invertible matrix $M\in GL(\R^{N})$ such that the transformed function
$$
  w(z):=v(y), \quad z:=M y,
$$
satisfies
\begin{equation}
	\left\{
	\begin{array}{ll}
		-\Delta w=0, \quad y\in \R^{N}, \\[1ex]
		0\leq w\leq w(0)=1.
	\end{array}
	\right.
\end{equation}
Thanks to the classical Liouville theorem, $w$ is constant and necessarily $w\equiv 1$, because $w(0)=1$. Therefore, $v\equiv 1$, which implies that,  for every $R>0$,
\begin{equation}
	\lim_{n\to+\infty} v_{n}=1 \;\; \hbox{in}\;\; \mc{C}^1(B_{R}).
\end{equation}
In particular, for every (fixed) $R>0$ and $\e\in (0,1)$, there exists $n_0\equiv n_0(R,\varepsilon)\geq 1$ such that
\begin{equation}
	\label{E1}
	\|v_{n}-1\|_{L^{\infty}(B_{R})}<\varepsilon, \quad \hbox{for all}\;\; n\geq n_0.
\end{equation}
By the definition of the rescaling function \eqref{K}, equation \eqref{E1} is equivalent to
\begin{equation}
	\label{E2} \|M_{n}^{-1}u_{n}(x_{n}+M_{n}^{\frac{1-r}{2}}\cdot)-1\|_{L^{\infty}(B_R)}<\varepsilon, \quad n\geq n_0.
\end{equation}
Setting $r_{n}:=M_{n}^{\frac{1-r}{2}}R$, equation \eqref{E2} can be rewritten as
\begin{equation*}
	\|M_{n}^{-1}u_{n}-1\|_{L^{\infty}(B_{r_{n}}(x_{n}))}<\varepsilon, \quad n\geq n_0.
\end{equation*}
In particular, the next lower bound holds
\begin{equation}
	\label{E3}
	M_{n}(1-\varepsilon)<u_{n}(x), \quad x\in B_{r_{n}}(x_{n}), \quad  n\geq n_0.
\end{equation}
On the other hand, since
$$
\mathscr{L}u_{n}=\l_{n}u_{n}+a(x)u_{n}^r\geq 0 \quad \hbox{on}\;\; \O_{+},
$$
it follows from Theorem \ref{SK}, applied in $\O_+$, that
\begin{equation}
\label{E4}
	u_{n}(x)\geq C  \, \|u_{n}\|_{L^{q}(\O_{+})} \, \dist(x,\partial\O_{+}) \quad
\hbox{for all}\;\; x\in\O_{+},
\end{equation}
where $C>0$ is a positive constant independent of $n$.
\par
To estimate the norm $\|u_{n}\|_{L^{q}(\O_{+})}$ we will exploit the local lower inequality \eqref{E3} by setting
\begin{align*}
	\|u_{n}\|_{L^{q}(\O_{+})}^{q}=\int_{\O_{+}}|u_{n}|^{q} \, {\rm{d}}x & \geq \int_{B_{r_{n}}(x_{n})\cap\O_{+}} |u_{n}|^{q} \, {\rm{d}}x \\
	& > M_{n}^{q} \, (1-\varepsilon)^{q} \, \mathbf{m}(B_{r_{n}}(x_{n})\cap\O_{+}).
\end{align*}
By hypothesis, $\O_{+}$ is a $\mc{C}^{1}$-subdomain of $\R^{N}$. Thus owing to
Theorem \ref{Th2.3}, there exists a constant $c>0$, independent of $n$, such that
$$
    \mathbf{m}(B_{r_{n}}(x_{n})\cap \O_{+}) \geq c r_{n}^{N} \quad \hbox{for all}\;\; n\in\N.
$$
Consequently, we obtain that
\begin{equation*}
	\|u_{n}\|^{q}_{L^{q}(\O_{+})}>c \, r_{n}^{N} M_{n}^{q}(1-\varepsilon)^{q},  \quad n\geq n_0,
\end{equation*}
or, equivalently,
\begin{equation}
	\label{E5}
	\|u_{n}\|_{L^{q}(\O_{+})}> c^{1/q} \, (1-\varepsilon) R^{N/q}\, M_{n}^{1+\frac{(1-r) N}{2q}}, \quad n\geq n_0.
\end{equation}
Combining the estimates \eqref{E4} and \eqref{E5}, we get \eqref{1.10}. So,
concluding the proof.

\section{Proof of Theorem \ref{APB}}\label{Se4}

We will argue by contradiction. Suppose that there exist sequences $\{(\l_{n},u_{n})\}_{n\in\mathbb{N}}\subset\mathscr{S}$ and $\{x_{n}\}_{n\in\mathbb{N}}\subset \bar{\Omega}$ such that
\begin{equation*}
	\|u_{n}\|_{L^{\infty}(\Omega)}=u_{n}(x_{n})\equiv M_{n}, \quad n\in\mathbb{N},
\end{equation*}
for some $x_n\in\bar\O$, $n\geq 1$, and
\begin{equation*}
	\lim_{n\to+\infty}\|u_{n}\|_{L^{\infty}(\Omega)}=+\infty.
\end{equation*}
Without lost of generality, we can assume  that $\{M_{n}\}_{n\in\N}$ is increasing and, owing to  Proposition \ref{ThPM}, that $x_n\in \bar \O_+$ for all $n\geq 1$. By compactness, since $\{x_{n}\}_{n\in\mathbb{N}}\subset \bar\Omega_+$ and $\{\l_{n}\}_{n\in \N}$ has been taken to be bounded, we can assume also that
\begin{equation*}
	\lim_{n\to+\infty}x_{n}=x_{0}, \quad \lim_{n\to +\infty}\l_{n}=\l_{0},
\end{equation*}
for some $x_{0}\in\bar{\Omega}_+$ and $\l_{0}\geq 0$. By Theorem \ref{ThFr}, necessarily, $x_{0}\in\partial\O_{+}$, and so  $a(x_{0})=0$.
\par
Subsequently, for every $L>0$, we set
$$
   [u_{n}>L]:=\{x\in\O_{+} \, : \, u_{n}(x) > L\}, \quad [u_{n}\leq L]:=\{x\in\O_{+} \, : \, u_{n}(x) \leq L\}.
$$
Under these assumptions, and keeping the same notations, the next result establishes that the blow-up must occur in the whole of $\O_{+}$.
\begin{lemma}
	\label{LC}
	For every $L>0$,
\begin{equation*}
		\lim_{n\to +\infty} \mathbf{m}([u_{n}\leq L])=0.
\end{equation*}
\end{lemma}
\begin{proof}
	 By Theorem \ref{BUS}, for every $1 \leq  q<\frac{N}{N-1}$, there are a positive constant $C= C(N,q)>0$, independent of $n$, and
an integer $n_0\in\N$ such that, for each $x\in [u_{n}\leq L]$,
\begin{equation*}
	C\, M_{n}^{1+\frac{N (1-r)}{2q}} \dist(x,\partial\O_{+}) \leq  u_{n}(x) \leq L, \qquad n\geq n_0.
	\end{equation*}
Thus,
\begin{equation*}
	\dist(x,\partial\O_{+})\leq  L  C^{-1} M_{n}^{\frac{N (r-1)}{2q}-1}, \qquad n\geq n_0.
\end{equation*}
Consequently, for every $n\geq n_0$,
\begin{equation}
	\label{Incl}
	[u_{n}\leq L]\subset \mc{F}_{n}:= \left\{x\in \O_{+} \, : \, \dist(x,\partial\O_{+})\leq L C^{-1} M_{n}^{\frac{N(r-1)}{2q}-1} \right\}.
\end{equation}
As we are assuming that $r<\frac{N+1}{N-1}$, it is apparent that
$$
  \tfrac{1}{2}(r-1)(N-1)<1.
$$
Pick any $\tfrac{1}{2}(r-1)(N-1)<\alpha<1$, and set
\begin{equation*}
	q=\frac{\alpha N}{N-1}.
\end{equation*}
Then,
\begin{equation*}
	\frac{(r-1)N}{2q}-1 < 0,
\end{equation*}
and hence,
\begin{equation*}
	\lim_{n\to+\infty}L \, C^{-1} \, M_{n}^{\frac{(r-1)N}{2q}-1}=0.
\end{equation*}
Therefore, by the monotonicity of the sequence $\{M_{n}\}_{n\in \N}$ and the definition of the sets $\{\mc{F}_{n}\}_{n\geq n_0}$, we find that
$$
   \mc{F}_{n_0}\supset \mc{F}_{n_0+1}\supset \cdots \supset \mc{F}_{n_0+n}\supset \cdots, \quad \hbox{and}\;\; \bigcap_{n\geq n_0}\mc{F}_{n}=\emptyset.
$$
Consequently, by the continuity from above of the Lebesgue measure, we can infer that
\begin{equation*}
	\lim_{n\to+\infty}\mathbf{m}(\mc{F}_{n})= \mathbf{m}\Big(\bigcap_{n\geq n_0}\mc{F}_{n}\Big)=0.
\end{equation*}
Finally, by \eqref{Incl} and the monotonicity of the measure,
\begin{equation}
	\lim_{n\to +\infty} \mathbf{m}([u_{n}\leq L]) \leq \lim_{n\to +\infty} \mathbf{m}(\mc{F}_{n})=0.
\end{equation}
This ends  the proof.
\end{proof}

\noindent Subsequently, we are denoting by
$$
   \s_{1}(a)\equiv \s_{1}[\mathscr{L},\O_+;a]>0
$$
the (positive) principal eigenvalue of the weighted eigenvalue problem
\begin{equation*}
	\left\{\begin{array}{ll}
		\mathscr{L} u = \s a(x) u & \hbox{in}\;\; \O_{+}, \\[1ex]
		u=0 &  \hbox{on}\;\; \p\O_{+}.
	\end{array}\right.
\end{equation*}
The existence and uniqueness of $\s_{1}(a)$ follows from the abstract theory of
\cite[Ch. 9]{LG13} taking into account that
$\sigma_{1}[\mathscr{L},\O_+]>0$, where $\sigma_{1}[\mathscr{L},\O_+]$ stands for the principal eigenvalue of
\begin{equation*}
	\left\{\begin{array}{ll}
		\mathscr{L} u = \s u & \hbox{in}\;\; \O_+, \\[1ex]
		u=0 & \hbox{on}\;\; \partial\O_{+}.
	\end{array}\right.
\end{equation*}
Now, let $L>0$ and $\varepsilon>0$ be two (arbitrary) positive constants such that
\begin{equation}
	\label{desf}
\sigma_{1}(a)=L^{r-1}-\varepsilon,
\end{equation}
and let  $\varphi\in \mc{C}^{2}(\bar{\O}_{+})$ be the unique positive eigenfunction associated to $\s_{1}(a)$ satisfying
$$
   \|\varphi\|_{L^{\infty}(\O_{+})}=1.
$$
Since $\varphi=0$ on $\partial\O_{+}$, integrating by parts in $\O_+$ shows that, for every $n\geq n_0$,
\begin{align*}
	\int_{\O_{+}}\left(\varphi  \mathscr{L} u_{n} - u_{n}\, \mathscr{L}\varphi\right) & = \int_{\partial\O_{+}} \langle A(x) \textbf{n}, \nabla \varphi\rangle   u_n \, {\rm d}\sigma - \int_{\partial\O_{+}} \langle A(x) \textbf{n}, \nabla u_n\rangle \varphi \, {\rm d}\sigma \\
	&=\int_{\partial\O_{+}} \langle A(x) \, \textbf{n}, \nabla \varphi\rangle   u_n \, {\rm d}\sigma.
\end{align*}
On the other hand, as $a(x)\gneq 0$ and $\sigma_{1}(a)>0$, $\varphi$ is a strict supersolution of $(\mathscr{L},\O_{+},\mf{D})$, i.e.,
\begin{equation*}
	\left\{\begin{array}{ll}
		\mathscr{L}\varphi = \s_{1}(a) a(x) \varphi \gneq 0 & \hbox{in}\;\; \O_{+}, \\[1ex] 		\varphi=0 & \hbox{on}\;\; \p\O_{+}.
	\end{array}\right.
\end{equation*}
Moreover, $h\equiv 1$ is a positive strict supersolution of $\mathscr{L}$ in $\O_+$, under Dirichlet boundary conditions, such that $h(x)>0$ for all $x\in\bar\O_+$.
Consequently, by \cite[Cor. 2.1]{LG13}, $\partial_{\nu}\varphi < 0$ for all $x\in\partial\O_{+}$, where $\nu$ is any outward pointing vector at $x$ for all $x\in \partial\O_{+}$. By the ellipticity of $A(x)$, we have that
$$
   \langle A(x) \textbf{n}, \textbf{n} \rangle > \mu |\textbf{n}|^2=\mu >0, \quad x\in \partial\O_{+}.
$$
Hence, the co-normal vector field $\nu(x):=A(x) \textbf{n}$ is an outward pointing vector at $x$ for all $x\in \partial\O_{+}$. So,
$$
   \langle A(x) \textbf{n}, \nabla \varphi\rangle = \partial_{\nu(x) }\varphi(x) < 0
    \quad \hbox{for all}\;\; x\in \partial\O_{+}.
$$
Therefore,
\begin{align*}
	\int_{\O_{+}}\left(\varphi  \mathscr{L} u_{n} - u_{n}\, \mathscr{L}\varphi\right) =\int_{\partial\O_{+}} \langle A(x) \textbf{n}, \nabla \varphi\rangle   u_n \, {\rm d}\sigma < 0\quad  \hbox{for all}\;\; n\geq n_0.
\end{align*}
Equivalently, for every $n\geq n_0$,
\begin{equation}
	\label{IN0}
  \int_{\O_{+}} (\l_{n} u_{n} + a(x) u_{n}^r) \varphi - \sigma_{1}(a)\int_{\O_{+}}  a(x) u_{n}\v = I_{1}+I_{2}< 0,
\end{equation}
where
\begin{align*}
	 I_{1} & := \int_{[u_n>L]} (\l_{n} u_{n} + a(x) u_{n}^r) \varphi - \sigma_{1}(a)\int_{[u_n>L]}  a(x) u_{n} \varphi,  \\[1ex]
	 I_{2} & :=\int_{[u_n\leq L]} (\l_{n} u_{n} + a(x) u_{n}^r)\varphi - \sigma_{1}(a)\int_{[u_n\leq L]}  a(x) u_{n} \varphi.
\end{align*}
Thanks to \eqref{desf}, for every $n\geq n_0$, we have that
\begin{align*}
	I_{1} & =\int_{[u_{n}> L]} \l_{n} u_{n}\v_n + \int_{[u_{n}> L]} (u_n^{r-1}-L^{r-1})  a u_{n} \varphi  + \varepsilon \int_{[u_{n}> L]} a u_{n}  \varphi \\
	& \geq \varepsilon \int_{[u_{n}> L]} a u_{n} \varphi \geq L \varepsilon \int_{[u_{n}> L]} a \varphi.
\end{align*}
On the other hand,
\begin{equation*}
	I_{2} \geq  - \sigma_{1}(a)   L \int_{[u_{n} \leq L]} a \varphi.
\end{equation*}
Thanks to \eqref{IN0}, we already know that $I_1+I_2<0$. Thus,
\begin{equation}
	\label{FI0}
	 L \varepsilon \int_{[u_{n}> L]} a \varphi - \sigma_{1}(a) L \int_{[u_{n} \leq L]} a \varphi<0 \quad
\hbox{for all}\;\; n\geq n_0.
\end{equation}
According to Lemma \ref{LC},
$$
  \lim_{n\to\infty} \mathbbm{1}_{[u_{n}>L]} = \mathbbm{1}_{\O_{+}} \;\; \hbox{pointwise almost everywhere
  in}\;\; \O_{+}.
$$
Moreover, $a \varphi\in L^{1}(\O_{+})$. Thus, owing to the dominated convergence theorem,
\begin{equation}
\label{ju1}
	\lim_{n\to+\infty} \int_{[u_{n}> L]} a(x) \varphi = \int_{\O_{+}} a(x) \varphi.
\end{equation}
On the other hand, since $0\leq \v\leq 1$, it is apparent that, for every $n\geq n_0$,
\begin{equation*}
	\left| \int_{[u_{n} \leq L]} a \varphi \right| \leq \|a\|_{L^{\infty}(\O_{+})} \mathbf{m}([u_{n}\leq L]).
\end{equation*}
Thus, thanks again to Lemma \ref{LC},
\begin{equation}
\label{ju2}
	\lim_{n\to+\infty}\left| \int_{[u_{n} \leq L]} a \varphi \right| =0.
\end{equation}
Therefore, letting $n\ua \infty$ in \eqref{FI0}, we find from \eqref{ju1} and \eqref{ju2} that
\begin{equation*}
  0< L \varepsilon \int_{\O_{+}} a \varphi \leq 0,
\end{equation*}
which is impossible. This contradiction ends the proof of Theorem \ref{APB}.

\end{document}